\documentclass{amsart}
\newcommand{\sv}[1]{}
 \newcommand{\lv}[1]{#1}
\usepackage{etoolbox}

\usepackage[utf8]{inputenc}
\usepackage[T1]{fontenc}
\lv{\usepackage{lmodern}}
\lv{\usepackage[foot]{amsaddr}}
\lv{\usepackage{fullpage}}
\usepackage{microtype}
\usepackage{amsfonts}
\usepackage{graphicx}
\usepackage{tabularx}
\usepackage{array}
\usepackage[usenames,dvipsnames]{color}
\usepackage{comment}
\usepackage{amsmath}
\lv{\usepackage{amsthm}}
\usepackage{amssymb}
\lv{\usepackage{fullpage}}
\usepackage{xcolor}
\usepackage{tikz}
\usetikzlibrary{backgrounds}
\usetikzlibrary{shapes.geometric}
\usetikzlibrary{calc}
\sv{\usepackage{sidecap}}
\usepackage{hyperref}
\usepackage[normalem]{ulem}
\usepackage{pifont}

\tikzstyle{legend_general}=[rectangle, rounded corners, thin,
                          top color= white,bottom color=lavander!25,
                          minimum width=2.5cm, minimum height=0.8cm,
                          violet]

\DeclareMathOperator{\cut}{\mathbin{\rotatebox[origin=c]{80}{\small\ding{34}}}}

\lv{
\newtheorem{theorem}{Theorem}[section]
\newtheorem{proposition}[theorem]{Proposition}
\newtheorem{conjecture}[theorem]{Conjecture}
\newtheorem{question}[theorem]{Question}

\newtheorem{lemma}[theorem]{Lemma}

\theoremstyle{definition}
\newtheorem{definition}[theorem]{Definition}
\newtheorem{example}[theorem]{Example}

}
\newtheorem{observation}[theorem]{Observation}
\DeclareMathOperator\df{:=}

\def\epsilon{\varepsilon}

\lv{
  \author{Peter Bradshaw$^\ddag$}
  \author{Tomáš Masařík$^\ddag$}
\author{Jana Novotn\'a$^\dagger$}
\author{Ladislav Stacho$^\ddag$}
\address{$\ddag$ Simon Fraser University, Burnaby, BC, Canada}
\address{$\dagger$ University of Warsaw, Warsaw, Poland}
\email{pabradsh@sfu.ca, masarik@kam.mff.cuni.cz,
janca@kam.mff.cuni.cz, ladislav\_stacho@sfu.ca}
\thanks{T.~Masařík was supported by a postdoctoral fellowship at the Simon Fraser University through NSERC grants R611450 and R611368. J.~Novotná was supported under the program of financial support for foreign internships for PhD students of the Warsaw Doctoral School of Mathematics and Computer Science. L.~Stacho was supported by NSERC grant R611368.}

\title{Robust Connectivity of Graphs on Surfaces} 
}
\sv{
\usepackage{url}

}

\begin{document}
\sv{
\mainmatter              
\title{Robust Connectivity of Graphs on Surfaces}
\titlerunning{Robust Connectivity of Graphs on Surfaces}  
\author{Peter Bradshaw\inst{1} \and Tomáš Masařík\inst{1}
Jana Novotná\inst{2} \and Ladislav Stacho\inst{1}}
\authorrunning{Peter Bradshaw et al.} %
\institute{Simon Fraser University, Burnaby, BC, Canada,\\
\email{pabradsh@sfu.ca, masarik@kam.mff.cuni.cz, ladislav\_stacho@sfu.ca}
\and
University of Warsaw, Poland
\email{janca@kam.mff.cuni.cz}
}
}
\maketitle

\begin{abstract}
  Let $\Lambda(T)$ denote the set of leaves in a tree $T$.
  One natural problem is to look for a spanning tree $T$ of a given graph $G$ such that $\Lambda(T)$ is as large as possible.
  \lv{This problem is called \textsc{maximum leaf number}, and it is a well-known NP-hard problem.
  Equivalently, the same problem can be formulated as the \textsc{minimum connected dominating set} problem, where the task is to find a smallest subset of vertices $D\subseteq V(G)$ such that every vertex of $G$ is in the closed neighborhood of $D$.}
  \lv{Throughout recent decades, these two equivalent problems have received considerable attention, ranging from pure graph theoretic questions to practical problems related to the construction of wireless networks.}%
  \lv{%

  }%
  \lv{  Recently, a similar but stronger notion was defined by Bradshaw, Masařík, and Stacho [Flexible List Colorings in Graphs with Special Degeneracy Conditions, ISAAC 2020].}%
  \sv{Recently, a similar but stronger notion called the \emph{robust connectivity} of a graph $G$ was introduced, which is}
  \lv{They introduced a new invariant for a graph $G$, called the \emph{robust connectivity} and written $\kappa_\rho(G)$, }%
  defined as the minimum value $\frac{|R \cap \Lambda (T)|}{|R|}$ taken over all nonempty subsets $R\subseteq V(G)$, where $T = T(R)$ is a spanning tree on $G$ chosen to maximize $|R \cap \Lambda(T)|$.
  \lv{Large robust connectivity was originally used to show flexible choosability in non-regular graphs.}%

  \lv{In this paper, we investigate some interesting properties of robust connectivity for graphs embedded in surfaces. }%
We prove a tight asymptotic bound of $\Omega(\gamma^{-\frac{1}{r}})$ for the robust connectivity of $r$-connected graphs of Euler genus $\gamma$.
Moreover, we give a surprising connection between the robust connectivity of graphs with an edge-maximal embedding in a surface and the \emph{surface connectivity} of that surface, which describes to what extent large induced subgraphs of embedded graphs can be cut out from the surface without splitting the surface into multiple parts. 
For planar graphs, this connection provides an equivalent formulation of a long-standing conjecture of Albertson and Berman\sv{.}\lv{ [A conjecture on planar graphs, 1979], which states that every planar graph on $n$ vertices contains an induced forest of size at least $n/2$.}
\sv{\keywords{Robust connectivity, graphs on surfaces, Albertson Berman's conjecture}}
\end{abstract}

\section{Introduction}
  Let $\Lambda(T)$ denote the set of leaves in a tree $T$.
  Given a graph $G$, we denote by $\tau_G$ the set of all spanning trees in $G$.
  The \emph{maximum leaf number} (or \emph{maxleaf number}) of a graph $G$ is defined as
  \lv{\[}%
    \sv{$}\ell(G)\df \max_{T\in \tau_G}| \Lambda(T)|.\sv{$}
\lv{\]}

Questions about maximum leaf number have been thoroughly considered throughout the literature, and \textsc{maximum leaf number} was one of the original NP-complete problems (even when restricted to planar graphs of maximum degree 4) identified by Garey and Johnson~\cite{GJ79}.
Storer \cite{Storer} considered the problem of finding a lower bound for the maximum leaf number of cubic graphs, and he proved that every cubic graph on $n$ vertices has a spanning tree with at least $\lceil \frac{n}{4} + 2 \rceil$ leaves. He also proved that this bound is sharp. 
\lv{Later, Griggs, Kleitman, and Shastri \cite{GriggsKleitman} proved that if a cubic graph on $n$ vertices is $3$-connected, then this lower bound can be improved to $\lceil \frac{n+4}{3} \rceil$. }%
\lv{Griggs and Wu \cite{GriggsWu} also showed that better lower bounds can be obtained for graphs of minimum degree $4$ or $5$. }%
Kleitman and West~\cite{KleitmanWest} gave an algorithm for a connected graph $G$ of minimum degree $k$ that shows $\ell(G)\ge(1-\frac{\sv{2.5}\lv{b}\ln{k}}{k})n$\sv{.}\lv{ for any constant $b>2.5$.}
  
The maximum leaf number problem can be equivalently formulated as a \emph{minimum connected dominating set} problem, which is a problem where the task is to find a smallest connected subset of vertices $D\subseteq V(G)$ of a graph $G$, such that every vertex of $G$ is in the closed neighborhood of $D$.
\sv{Both formulations of the maximum leaf number problem have been studied from the computational point of view in many areas of computer science~\cite{DW13}. }%
\lv{%
Both formulations of the maximum leaf number problem have been considered and studied from the computational point of view in many areas of computer science, including approximation-algorithms~\cite{Bonsma2011,Bonsma2012} and exact enumeration algorithms~\cite{Fomin2007,Fernau2011,Lokshtanov2018}.
Some of these research directions are motivated by a strong connection with the construction of wireless networks; consult the following book~\cite{DW13} and the survey~\cite{Du13} for more details on this vast topic.
}%
\lv{In Fellows~\cite{Fellows09}, the maxleaf number is used for a construction of an efficient parameterization for solving some basic problems, including the 3-coloring and Hamiltonian path problems.}

In this paper, we will consider a graph invariant related to maximum leaf number known as \emph{robust connectivity}\footnote{
  This parameter was formerly called game connectivity in~\cite{Flexibility-ISAAC}\lv{ and older versions of~\cite{Flexibility-arxiv}}. However, we believe that the term ``robust connectivity'' better suits the properties of this parameter.}.
The robust connectivity $\kappa_\rho(G)$ of a graph $G$ is defined as follows. 
 \begin{definition}[Robust connectivity~\cite{Flexibility-ISAAC}]
\label{def:game_connectivity}
\lv{\[}%
  \sv{$}%
  \kappa_\rho(G) \df \min_{\substack{R \subseteq V(G) \\ R \neq \emptyset}} \max_{T\in\tau_G} \frac{|R \cap \Lambda (T)|}{|R|}.
  \sv{$}
\lv{\]}%
\end{definition}
We often write $\ell(G,R)$ for the maximum value of $\frac{|R \cap \Lambda (T)|}{|R|}$ taken over all spanning trees $T$ of $G$, in which case $\kappa_\rho(G) = \min_{\substack{R \subseteq V(G) \\ R \neq \emptyset}} \ell(G,R)$.
\lv{%
  We may think of robust connectivity in terms of a one-turn game in which the first player chooses a set $R$ of vertices in $G$, and then the second player attempts to find a spanning tree in $G$ using as many vertices of $R$ as leaves as possible; see~\cite{Flexibility-arxiv} for details.
This one-turn can be also compared with the one-turn matching game used by Matuschke, Skutella, and Soto to define robust matchings~\cite{Skutella}.
}%
\lv{The notion of robust connectivity was first used in the context of flexible list colorings~\cite{Flexibility-ISAAC}.
  The \textsc{list coloring} problem is a well-known problem in which the task is to give a graph $G$ a proper coloring, called a \emph{list coloring}, in which every vertex $v \in V(G)$ uses a color from some predetermined list $L(v)$.
  The function $L$ is called a \emph{list assignment}, and the \emph{size} of a list assignment $L$ is defined as the minimum value $|L(v)|$ taken over all $v \in V(G)$.
In the flexible list coloring problem, 
we again have a graph $G$ and a list $L$ of colors at each vertex, but we also have certain vertices $v$ for which some color in $L(v)$ is \emph{preferred}; then, our task is to find a proper list coloring on $G$ that satisfies as many of these coloring preferences as possible. For a value $\epsilon > 0$, we say that a graph $G$ is \emph{$\epsilon$-flexibly $k$-choosable} if for any list assignment $L$ of size $k$, we can always find a list coloring that satisfies at least an $\epsilon$ proportion of any set of coloring preferences. }%
In \cite{Flexibility-ISAAC}, it was shown that a non-regular graph $G$ of maximum degree $\Delta$ is $\frac{\kappa_\rho(G)}{2 \Delta} $-flexibly $\Delta$-choosable.

Despite being useful for establishing bounds in certain problems like flexible list coloring, robust connectivity does not appear to be simple to calculate. 
However, in \cite{Flexibility-ISAAC}, it was shown that for graphs of bounded degree, $3$-connectivity is enough to guarantee some absolute lower bound for robust connectivity.

\begin{theorem}[Theorem 22 in \cite{Flexibility-ISAAC}\lv{\footnote{This theorem appears as Theorem 5.6~in the full version on arXiv~\cite{Flexibility-arxiv}.}}]
\label{thm:Flex_gc}
If $\Delta \geq 3$ is an integer, then there exists a value $\epsilon = \epsilon(\Delta) > 0$ such that if $G$ is a $3$-connected graph of maximum degree $\Delta$, then $\kappa_\rho(G) \geq \epsilon$. 
\end{theorem}

The authors of \cite{Flexibility-ISAAC} also showed that $3$-connectivity alone is not enough to guarantee a lower bound on a graph's robust connectivity. 
To demonstrate this fact, the authors used the Levi graph of the complete $3$-uniform hypergraph $K_n^{(3)}$, described in Example~\ref{ex:highConnectivity}.
They also pointed out that $2$-connected cubic planar graphs do not have any guaranteed nonzero lower bound for robust connectivity, as demonstrated by Figure~\ref{fig:diamonds}.

\sv{\begin{SCfigure}[2][t]}
  \lv{
    \begin{figure}
    \begin{center}
    }
\begin{tikzpicture}
  [scale=\lv{1.7}\sv{1},auto=left,every node/.style={circle,fill=gray!30,minimum size = 6pt,inner sep=0pt}]
\node (a1) at (1.6,0.4) [draw = black] {};
\node (a2) at (1.6,-0.4) [draw = black, fill = black] {};
\node (a3) at (1.8,0) [draw = black] {};
\node (a4) at (1.4,0) [draw = black] {};

\node (b1) at (-1.6,0.4) [fill = white] {};
\node (b2) at (-1.6,-0.4) [fill = white] {};

\node(dots) at (-1.6,0) [fill = white] {$\cdots$};

\node(c1) at (0.49,1.6) [draw = black] {};
\node(c2) at (1.16,1.2) [draw = black, fill = black] {};
\node(c3) at (0.7,1.2) [draw = black] {};
\node(c4) at (0.95,1.6) [draw = black] {};

\node(d1) at (-0.49,1.6) [draw = black, fill = black] {};
\node(d2) at (-1.16,1.2) [draw = black] {};
\node(d3) at (-0.7,1.2) [draw = black] {};
\node(d4) at (-0.95,1.6) [draw = black] {};

\node(e1) at (0.49,-1.6) [draw = black, fill = black] {};
\node(e2) at (1.16,-1.2) [draw = black] {};
\node(e3) at (0.7,-1.2) [draw = black] {};
\node(e4) at (0.95,-1.6) [draw = black] {};

\node(f1) at (-0.49,-1.6) [draw = black] {};
\node(f2) at (-1.16,-1.2) [draw = black, fill = black] {};
\node(f3) at (-0.7,-1.2) [draw = black] {};
\node(f4) at (-0.95,-1.6) [draw = black] {};

\foreach \from/\to in {a1/a3,a1/a4,a3/a4,a3/a2,a4/a2,c2/c3,c2/c4,c3/c4,c3/c1,c4/c1,d2/d3,d2/d4,d3/d4,d3/d1,d4/d1,e2/e3,e2/e4,e3/e4,e3/e1,e4/e1,f2/f3,f2/f4,f3/f4,f3/f1,f4/f1,a1/c2,c1/d1,f1/e1,e2/a2, d2/b1,b2/f2}
    \draw (\from) -- (\to);
\end{tikzpicture}
\lv{\end{center}}
\caption{The graph $G$ in the figure is an arbitrarily large two-connected $3$-regular graph. If a set $R \subseteq V(G)$ is chosen as shown by the dark vertices in the figure, then there does not exist a constant $\epsilon > 0$ such that $\epsilon|R|$ vertices of $R$ may become leaves of some spanning tree of $G$. Therefore, the robust connectivity of $G$ is arbitrarily close to zero. For any $k \geq 3$, a similar $k$-regular graph with robust connectivity arbitrarily close to zero with may be constructed from a cycle $C$ by replacing each vertex of $C$ by a $k$-clique minus an edge.}
\label{fig:diamonds}
\lv{\end{figure}}
\sv{\end{SCfigure}
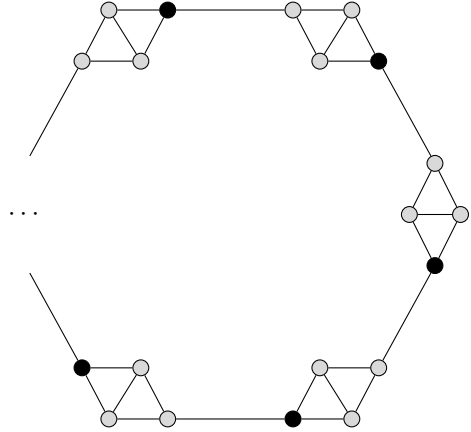}

\begin{example}[\cite{Flexibility-ISAAC}]\label{ex:highConnectivity}
Let $G$ be a graph whose vertex set consists of a set $R$ of at least four vertices and an additional vertex $v_A$ for each triplet $A \in {R \choose 3}$, and let each vertex of the form $v_A$ be adjacent exactly to those vertices in the triplet $A$. 
\end{example}
It is straightforward to show that the Levi graph $G$ of $K_n^{(3)}$ in Example \ref{ex:highConnectivity} is 3-connected.
However, no more than two vertices of $R$ may be removed from $G$ without disconnecting $G$. Therefore, for any spanning tree $T$ on $G$, the leaves of $T$ include at most two vertices of $R$. As $R$ becomes arbitrarily large, the proportion of vertices in $R$ that can be included as leaves in a spanning tree on $G$ becomes arbitrarily small. 
Therefore, Example~\ref{ex:highConnectivity} shows that some 3-connected graphs $G$ do not satisfy $\kappa_\rho(G) \geq \epsilon$ for any universal $\epsilon > 0$. 
Figure \ref{fig:counterexample} shows the Levi graph of $K_n^{(3)}$ from Example \ref{ex:highConnectivity} when $n = |R| = 5$.
Similarly, the Levi graphs of $K_n^{(r)}$ for larger uniformities $r \geq 4$ show that the robust connectivity of $r$-connected graphs may also be arbitrarily small.

\sv{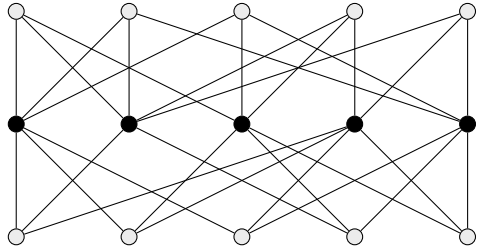
\begin{SCfigure}[2][t]}
  \lv{  \begin{figure}
    \begin{center}}
\begin{tikzpicture}
  [scale=\lv{1.5}\sv{1},auto=left,every node/.style={circle,fill=gray!15,minimum size = 6pt,inner sep=0pt}]
\node (r1) at (0,0) [draw = black, fill = black] {};
\node (r2) at (1,0) [draw = black, fill = black] {};
\node (r3) at (2,0) [draw = black, fill = black] {};
\node (r4) at (3,0) [draw = black, fill = black] {};
\node (r5) at (4,0) [draw = black, fill = black] {};
\node (u1) at (0,1) [draw = black, fill = gray!15] {};
\node (u2) at (1,1) [draw = black, fill = gray!15] {};
\node (u3) at (2,1) [draw = black, fill = gray!15] {};
\node (u4) at (3,1) [draw = black, fill = gray!15] {};
\node (u5) at (4,1) [draw = black, fill = gray!15] {};
\node (d1) at (0,-1) [draw = black, fill = gray!15] {};
\node (d2) at (1,-1) [draw = black, fill = gray!15] {};
\node (d3) at (2,-1) [draw = black, fill = gray!15] {};
\node (d4) at (3,-1) [draw = black, fill = gray!15] {};
\node (d5) at (4,-1) [draw = black, fill = gray!15] {};

\foreach \from/\to in {u1/r1,u1/r2,u1/r3,d1/r1,d1/r2,d1/r4,u2/r1,u2/r2,u2/r5,d2/r1,d2/r3,d2/r4,u3/r1,u3/r3,u3/r5,d3/r1,d3/r4,d3/r5,u4/r2,u4/r3,u4/r4,d4/r2,d4/r3,d4/r5,u5/r2,u5/r4,u5/r5,d5/r3,d5/r4,d5/r5}
    \draw (\from) -- (\to);

\end{tikzpicture}
\lv{\end{center}}
\caption{[Figure~4 in~\cite{Flexibility-ISAAC}] The figure shows the Levi graph of $K_5^{(3)}$, which is a 3-connected graph constructed based on Example~\ref{ex:highConnectivity} with $|R|=5$.
}
\label{fig:counterexample}
\lv{\end{figure}}
\sv{\end{SCfigure}}

We show a surprising connection between the notion of robust connectivity and the following famous conjecture of Albertson and Berman~\cite{AlbertsonBerman}\sv{.}\lv{, which states that every planar graph on $n$ vertices contains an induced forest of size at least $n/2$.}
\begin{conjecture}[\cite{AlbertsonBerman}]
\label{conj:AB}
If $G$ is a planar graph on $n$ vertices, then $G$ contains an induced forest of size at least $n/2$.
\end{conjecture}
\lv{Albertson and Berman's conjecture can be equivalently formulated using the notion of a \emph{feedback vertex set}, which is a set of vertices $X\subseteq V(G)$ such that $G\setminus X$ does not contain any cycle. With this definition,
Conjecture \ref{conj:AB} equivalently states that every planar graph $G$ on $n$ vertices has a feedback vertex set of size at most $n/2$.}
Conjecture \ref{conj:AB} has a long history, and many partial results and theorems of a similar flavor exist; see \cite{bonamy2020fractional} for a very recent overview of the related results.
One of the large motivations for Conjecture~\ref{conj:AB} was that it would provide a proof that every planar graph on $n$ vertices has an independent set of size $\lceil n/4 \rceil$ without relying on the Four Color Theorem.
The currently best known lower bound of $\frac{2}{5}n$ is a consequence \lv{(see~\cite{Fertin02} for details)} of Borodin's theorem of 5-acyclic colorability~\sv{\cite{Borodin}}\lv{\cite{Borodin,Borodin_eng}}%
, which was already published in 1976. 
Conjecture \ref{conj:AB} is proven for only a few subclasses of planar graphs\sv{, e.g.,}\lv{:} outerplanar graphs~\cite{Hosono90}, where the tight lower bound is $\frac{2}{3}n$\sv{.}\lv{, planar triangle-free graphs~\cite{Salavatipour,Dross-triangle}, and planar graphs of girth~5~\cite{Kelly,Shi}.}
\lv{Recently, a great effort was dedicated to find a largest induced \emph{linear forest}, that is, a disjoint union of paths.
A series of papers by different authors was concluded in~\cite{Dross2019} where it is shown that triangle-free planar graph on $n$ vertices and $m$ edges has an induced linear forest with at least $\frac{9n-2m}{11}$ vertices.}

\lv{\subsection{Our Results}}
\sv{\medskip\noindent\textbf{Our results.}}
We present asymptotically tight lower bounds for the robust connectivity of $r$-connected graphs in terms of their Euler genus, for $r\ge 3$.
Recall that Figure~\ref{fig:diamonds} shows that $2$-connected planar graphs do not have any nonzero lower bound.

\begin{theorem}\label{thm:genus-gc-asy}
  If $r\geq 3$ and $G$ is an $r$-connected graph of Euler genus $\gamma$, then $\kappa_\rho(G) \geq \frac{1}{27}\gamma^{-1/r}$.
\end{theorem}

Jing and Mohar~\cite{JingMohar} proved that the Euler genus of the Levi graph of $K_n^{(3)}$ equals $\frac{(n-2)(n+3)(n-4)}{12}$ for even values $n\geq 6$, and in general, it follows straightforwardly from Euler's formula that the Euler genus of $K_n^{(r)}$ is $\Theta( n^r)$; check Theorem~\ref{thm:constuction} for more details. 
Therefore, the Levi graphs of $K_n^{(r)}$, for which robust connectivity is at most $\frac{r-1}{n}$, show that the bound of Theorem \ref{thm:genus-gc-asy} is tight within a constant factor for each fixed $r \geq 3$.
With more careful calculations, we also derive improved lower bounds for the robust connectivity of 3-connected planar graphs.
\lv{At this point we wish to remark that even for planar 3-connected graphs, robust connectivity and maxleaf number might differ considerably.
  Take, for example, the family of graphs as depicted and described in Figure~\ref{fig:biggerblowup}.
  We will show that the robust connectivity of these graphs is at most $\frac{1}{3}$ plus small additive constant (see Theorem~\ref{thm:UB}), but their maxleaf number is greater than $\frac{n}{2}$.

\begin{figure}
  \begin{center}
\begin{tikzpicture}
  [scale=1,auto=left,every node/.style={circle, draw=black, fill=gray!30, inner sep=2pt}, black node/.style={circle, draw=black, fill=black, inner sep=3pt}, spanning edge/.style={line width=1.5pt, draw=red}, nonspanning/.style={line width=1pt, draw=black}, leaf/.style={draw=blue, line width=1.5pt}]

\node[draw=none,fill=none,minimum size=3cm,regular polygon,regular polygon sides=6,rotate=60] (a0)   {};
\node[nonspanning,fill=none,minimum size=2cm,regular polygon,regular polygon sides=6,rotate=60] (aa0)  {};

\foreach \n/\move in {90/4,210/4,330/4}{
\node[draw=none,fill=none,minimum size=3cm,regular polygon,regular polygon sides=6] (a\n)  at (\n:\move) {};
\node[nonspanning,fill=none,minimum size=2cm,regular polygon,regular polygon sides=6] (aa\n)  at (\n:\move) {};
 \node[black node] (i\n) at (\n:2) {}; 
 \node[black node,leaf] (j\n) at (\n+60:5.7) {}; 
}
\foreach \n in {0,90,210,330}{
 \foreach \x in {1,2,...,6}{
  \node[every node] (a\n\x) at (a\n.corner \x) {};
  \node[every node, leaf] (aa\n\x) at (aa\n.corner \x) {};
  \draw[spanning edge] (a\n\x) -- (aa\n\x);
    }
}
\foreach \n in {0,90,210,330}{
\foreach \from/\to in {1/2,2/3,3/4,4/5,5/6,6/1}{
 \node[every node,leaf] (aux1) at ($(a\n\from)!0.33!(a\n\to)$) {};
 \node[every node,leaf] (aux2) at ($(a\n\from)!0.66!(a\n\to)$) {};
 \node[every node,leaf] (aux3) at ($(aa\n\from)!0.33!(aa\n\to)$) {};
 \node[every node,leaf] (aux4) at ($(aa\n\from)!0.66!(aa\n\to)$) {};
 \draw[nonspanning] (aux1) -- (aux3);
 \draw[nonspanning] (aux2) -- (aux4);
 \draw[nonspanning] (aux2) -- (aux1);
 \draw[spanning edge] (a\n\from) -- (aux3);
 \draw[spanning edge] (a\n\to) -- (aux4);
 \draw[spanning edge] (a\n\to) -- (aux2);
 \draw[spanning edge] (a\n\from) -- (aux1);

}

 \draw [spanning edge,bend right] (a\n1) edge (a\n2);
 \draw [spanning edge,bend right] (a\n2) edge (a\n3);
 \draw [spanning edge,bend right] (a\n3) edge (a\n4);
 \draw [spanning edge,bend right] (a\n4) edge (a\n5);
 \draw [spanning edge,bend right] (a\n5) edge (a\n6);
 \draw [bend right,nonspanning] (a\n6) edge (a\n1);
}
\foreach \from/\cor/\to in {90/5/90,0/1/90, 210/6/210,0/2/210, 330/3/330, 0/5/330}
   { 
       \draw[spanning edge] (a\from\cor) -- (i\to); 
   }
\foreach \from/\cor/\to in {90/4/90,0/6/90, 210/1/210,0/3/210,330/2/330,0/4/330}
   { 
       \draw[nonspanning] (a\from\cor) -- (i\to); 
   }
\foreach \from/\cor/\to in {90/2/90,90/1/330,90/6/330,210/4/210,210/3/90,210/2/90,330/6/330,330/4/210, 330/5/210}
   { 
       \draw[nonspanning] (a\from\cor) -- (j\to);
   }
\foreach \from/\cor/\to in {90/3/90,210/5/210,330/1/330}
  { 
      \draw[spanning edge] (a\from\cor) -- (j\to);
  }

\end{tikzpicture}
\end{center}
\caption{
The figure shows an instance of a construction of a graph $G$ where vertices of an arbitrary cubic graph $H$ (here $K_4$) are replaced by a modified circular ladder on 36 vertices (depicted in gray), and the edges of $H$ are replaced by vertices of degree 4 (depicted in black), with neighbors as shown in the example.
  In Theorem~\ref{thm:UB}, we will show that as the size of the cubic graph $H$ used in the construction grows, the robust connectivity of $G$ approaches $\frac{1}{3}$.
  However, here the maxleaf number of $G$ is at least $\frac{61}{75}|V(G)|$.
  A spanning tree achieving this maxleaf number is shown here in red, with its leaves in blue. 
}
\label{fig:biggerblowup}
\end{figure}
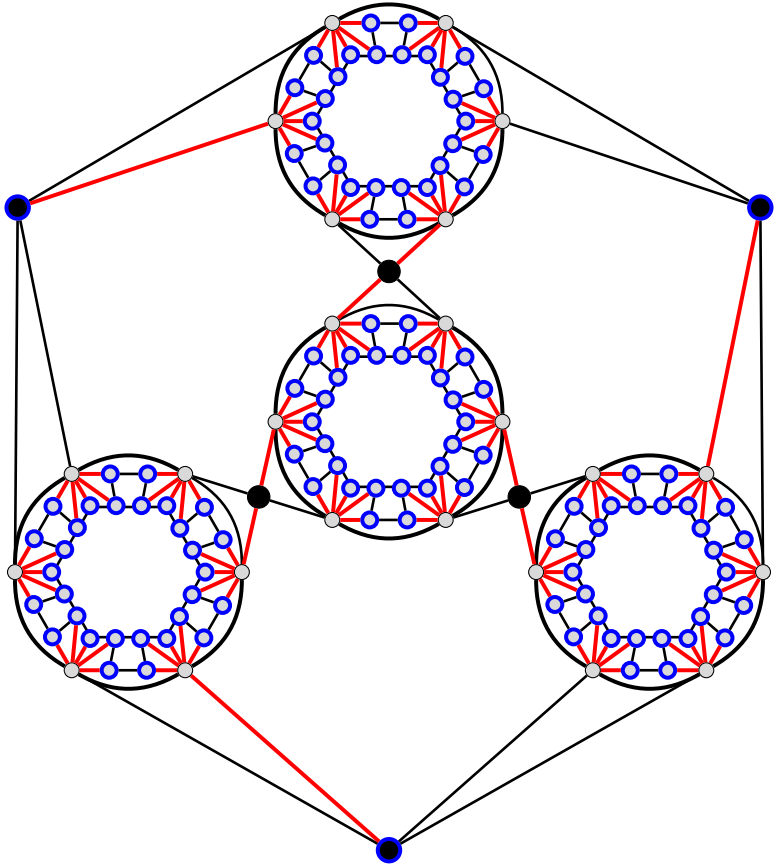
}
\begin{theorem}\label{cor:planar3con}
  If $G$ is a $3$-connected planar graph, then $\kappa_\rho(G) \geq \frac{21}{256}>\frac{1}{13}$. 
  Moreover, if $\epsilon > 0$, then there exists a planar $3$-connected graph $H$ such that $\kappa_\rho(H) \leq \frac{1}{3}+\varepsilon$.
\end{theorem}

In this direction, we may attempt to go even further and exchange the assumption of 3-connectivity with being a planar triangulation.
Note that planar triangulations on at least $4$ vertices are $3$-connected.
For planar triangulations, we formulate the following conjecture.
\begin{conjecture}
\label{conj:K}
If $G$ is a planar triangulation, then $\kappa_\rho(G) \geq \frac{1}{2}$.
\end{conjecture}

\lv{
  On a similar note, it was shown already in 1990 by Albertson, Berman, Hutchinson, and Thomassen~\cite{Albertson90} that for a planar triangulation $G$ on at least 4 vertices, there is always a spanning tree without degree 2 nodes, which yields a tree that has at least $\frac{1}{2}|V(G)|$ leaves.
  In fact, the same is true for any triangulation of any surface, as shown by Chen, Ren, and Shan~\cite{Chen12}.
  Therefore, the maxleaf number is at least $\frac{1}{2}|V(G)|$ for any graph $G$ that is a triangulation of some surface.
A trivial upper bound of $\frac{2}{3}|V(G)|$ for the maxleaf number of a planar triangulation is given by a triangle and also by the icosahedron, but notably we have not found any better bound. 
}

Surprisingly, Conjecture~\ref{conj:K} turns out to be equivalent to the famous Conjecture~\ref{conj:AB}.
\begin{theorem}\label{thm:conjecturs}
Conjecture \ref{conj:AB} is equivalent to Conjecture \ref{conj:K}.
\end{theorem}

Hence, we propose the notion of robust connectivity as another way to attack Conjecture~\ref{conj:AB}.
In fact, we will present a further generalization of both conjectures to graphs of arbitrary Euler genus $\gamma$.
In order to do this, we develop a new notion for an arbitrary surface $S$ that, informally, describes how large of an induced subgraph we can cut out of an edge-maximal graph on $S$ without separating $S$ into multiple pieces.

We will often write $\Tilde G$ to refer to an embedding of graph $G$.
For a surface $S$, we let $\mathcal G_S$ be the family of (simple) embedded graphs on $S$. 
A graph $G$ is \emph{edge-maximal} (with respect to a surface $S$) if $G$ has an embedding $\Tilde G \in \mathcal G_S$, but for each non-edge $e\notin E(G)$, $G + e$ cannot be embedded in $S$.
More often, we will speak about an \emph{edge-maximal embedding} $\Tilde G\in \mathcal G_S$ of graph $G$
if for each $e\notin E(G)$, $e$ cannot be added to the embedding $\Tilde G$ without creating a crossing on $S$.
Note that every edge-maximal graph $G$ with respect to a surface $S$ has an edge-maximal embedding on $S$, but not every graph with an edge-maximal embedding on $S$ is edge-maximal with respect to $S$. In particular, McDiarmid and Wood \cite{McDiarmid2018} show that for any surface $S$, there exist infinitely many planar graphs with an edge-maximal embedding in $S$. In Figure~\ref{fig:edgeMaxPlanar}, we show an example of a planar graph that is not edge-maximal with respect to the projective plane but that has an edge-maximal embedding in the projective plane.
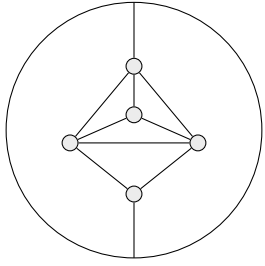
\begin{figure}
\begin{tikzpicture}
[scale=1.7,auto=left,every node/.style={circle,fill=gray!15,minimum size = 6pt,inner sep=0pt}]

\draw (0,0) circle (1);

\node (v1) at (-0.5,-0.1) [draw = black, fill = gray!15] {};
\node (v2) at (0.5,-0.1) [draw = black, fill = gray!15] {};
\node (v3) at (0,0.5) [draw = black, fill = gray!15] {};
\node (v4) at (0,-0.5) [draw = black, fill = gray!15] {};
\node(v5) at (0,0.12) [draw = black] {};
\foreach \from/\to in {v1/v2,v1/v3,v1/v4,v2/v3,v2/v4,v5/v1,v5/v2,v5/v3}
    \draw (\from) -- (\to);
    \draw (v3) -- (0,1);
    \draw (v4) -- (0,-1);
\end{tikzpicture}
\caption{The figure shows the planar graph $K_5 - e$ with an edge-maximal embedding in the projective plane. Note that the graph $K_5 - e$ is not edge-maximal with respect to the projective plane, since $K_5$ has a projective plane embedding.}
\label{fig:edgeMaxPlanar}
\end{figure}

\lv{Recently, edge-maximally embedded graphs on a surface $S$ of Euler genus $\gamma$ were shown to be at most $O(\gamma)$ edges short of a triangulation on $S$ \cite{McDiarmid2018}, which solved a long-standing open problem of Kainen~\cite{Kainen}.
In \cite{Pfender}, Davies and Pfender 
constructed an infinite family of edge-maximal graphs $G \in \mathcal G_S$ for orientable surfaces $S$ of Euler genus $\gamma\geq 4$ that are $\Theta(\gamma)$ edges short of a triangulation.
}%

For an embedded graph $\Tilde G \in \mathcal G_S$,
we write $S \cut \Tilde G$ for the surface obtained by cutting $S$ along the edges of $\Tilde G$ and puncturing $S$ at each isolated vertex of $\Tilde G$.
Then, we define $m(\Tilde G)$ to be the number of vertices in a largest induced embedded subgraph $\Tilde{G'} \subseteq \Tilde G$ for which $S \cut \Tilde{G'}$ is a connected surface.
\lv{%
In Figure~\ref{fig:TorusCurve}, we give an example of an embedded graph $\Tilde G$ in a surface $S$ with an induced embedded subgraph $\Tilde {G'}$ for which $S \cut \Tilde{G'}$ is a connected surface. }%
Now, we define the main parameter of interest.

\lv{
\lv{\begin{figure}}
  \sv{\begin{SCfigure}[2]t]}
  \begin{tikzpicture}[scale=\lv{1}\sv{0.7}, shading=ball, ball color=lightgray, every node/.style={circle,draw = black, fill=gray,inner sep=2pt}]
    \draw[lightgray, fill=lightgray] ellipse (3 and 1.65);

    \draw (-0.8,1.1) .. controls (0.8,1.1) and (-0.2,-.5) .. (0,-0.8);   
       \node (z) at (-1.2,1.2) [draw = none, fill=none]{$v$};
       \node (z) at (-1.9,0) [draw = none, fill=none]{$w$};
       \node (z) at (-1.6,-0.8) [draw = none, fill=none]{$x$};
       \node (z) at (0.4,-0.9) [draw = none, fill=none]{$y$};
       
    \node (A) at (-0.8,1.1) {};
    \node (B) at (0,-0.8) {};
    \node (C) at (-1.5,0) {};
    \node (D) at (-1.2,-0.7) {};

   \begin{scope}[shift={(-0.4,0)}]
   \draw[dashed] (-0.5,0)  coordinate(c1) {[yscale=1] arc(180:0:.2)} coordinate(c2);
    \draw[dashed ](1,0)  coordinate(d1) {[yscale=1] arc(180:0:.2)} coordinate(d2);
   
    \draw[fill opacity=.8,shade]
      (c1) to[out=60,in=120,looseness=1.6] (d2)
      {[yscale=.5] arc(360:180:.2)}
      to[out=120,in=60,looseness=1.6] (c2)
      {[yscale=.5] arc(360:180:.2)} -- cycle;
   \end{scope}
   \foreach \from/\to in {A/C,B/C,B/D,C/D}
       \draw (\from) -- (\to);
  \end{tikzpicture}
  \caption{The figure shows an embedded diamond graph $\Tilde G$ in a surface~$S$ homeomorphic to the torus. When $\Tilde G'= \Tilde G[v,w,y]$ is the embedded subgraph induced by $v,w,y$, $S \cut \Tilde G'$ is a connected surface; however, when $\Tilde G' = \Tilde G[w,x,y]$, $S \cut \Tilde G'$ is not a connected surface.}
  \label{fig:TorusCurve}
\sv{\end{SCfigure}}
\lv{\end{figure}
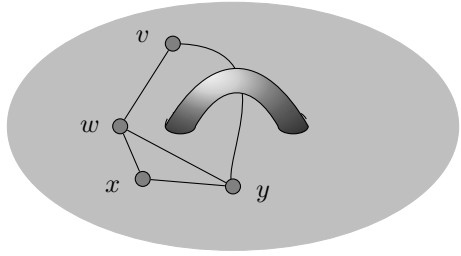}
}

\begin{definition}[Surface connectivity]
  The \emph{surface connectivity} ($\kappa_s(S)$) of a surface $S$ is defined as follows:
  \lv{\[ }%
  \sv{$}%
\kappa_s(S) = \inf \left \{\frac{m(\Tilde G)}{|\Tilde G|} : \Tilde G \in \mathcal G_S \right  \}. 
  \sv{$}%
\lv{\]}
\end{definition}

We will show that the minimum robust connectivity over all edge-maximal graphs embedded in a surface $S$ is equal to the surface connectivity of $S$.
\begin{theorem}
\label{thm:SgTriangulation}
Let $S$ be a surface. Every graph $G$ with an edge-maximal embedding on $S$ satisfies $\kappa_\rho(G) \geq k$ if and only if
$\kappa_s(S) \geq k$.
\end{theorem}

When $S$ is the plane, an edge-maximally embedded graph $\Tilde G$ is simply a planar triangulation, and $S \cut \Tilde G'$ is connected for an induced embedded subgraph $\Tilde G' \subseteq G$ if and only if $\Tilde G'$ is acyclic.
Hence, Theorem~\ref{thm:SgTriangulation} directly implies Theorem~\ref{thm:conjecturs}. 

\lv{
\medskip
\noindent
\textbf{Structure of the paper.} 
We first present a proof of Theorem~\ref{thm:SgTriangulation} in Section~\ref{sec:edgeMaximal}.
Then, in Section~\ref{sec:gcgenus}, we show the proofs of Theorems~\ref{thm:genus-gc-asy} and \ref{cor:planar3con}.
}

\sv{Due to space limitations we will provide the proofs of Theorems~\ref{thm:genus-gc-asy} and \ref{cor:planar3con} in the full version which will be available on arXiv shortly.} 

\lv{\subsection{Preliminaries}}

\lv{
For a function $f$, we define the \emph{right-hand derivative} of $f$ as follows:
\[f_+'(t) = \lim_{h \rightarrow 0^+} \frac{f(t+h)- f(t)}{h}.\]

We refer to a book of Mohar and Thomassen~\cite{MT-book} for an introduction to graphs on surfaces as well as for standard terminology that we use in our paper. For basic topological definitions and concepts, we refer to a book by Hatcher \cite{Hatcher}.
All graphs that we consider are without loops and parallel edges.
For a vertex $v$, we use $N(v)$ for the open neighborhood of $v$, which does not include $v$.

A \emph{closed surface} is a compact Hausdorff topological space in which every point has an open neighborhood that is homeomorphic to an open disk in the plane. 
A \emph{surface} then is a topological space $S$ that is obtained from a closed surface by removing finitely many path-connected open sets $D_1, \dots, D_k$ with disjoint boundaries. We say that the points on the boundaries of these sets $D_i$ form the \emph{boundary} of $S$, and we write $S^{\circ}$ for the \emph{interior} of $S$, consisting of all non-boundary points.
A surface is \emph{triangulated} if it can be realized as the union of a finite set of triangles with some of their vertices and edges identified. It is well-known (\cite[Theorem 3.1.1]{MT-book}) that every surface is homeomorphic to a triangulated surface.

A \emph{path} on a surface $S$ is a continuous function $h:[0,1] \rightarrow S^{\circ}$.
A surface $S$ is \emph{connected} if for any two points $x, y \in S^{\circ}$, there exists a path $h : [0,1] \rightarrow S^{\circ}$ for which $h(0) = x$ and $h(1) = y$. While this definition of connectedness is slightly stronger than the traditional definition of connectedness for general topological spaces, our definition of connectedness is equivalent to the traditional definition when restricted to surfaces.
For a connected triangulated closed surface $T$ with $n$ vertices, $m$ edges, and $f$ triangular faces, the \emph{Euler genus} $\gamma(T)$ of $T$ is defined as $2 - n + m - f$. For a connected surface $S$, we say that the Euler genus of $S$ is equal to $\gamma(T)$, where $T$ is a triangulated closed surface of minimum genus for which $S$ is homeomorphic to a surface obtained from $T$ by removing finitely many path-connected open sets with disjoint boundaries. It is shown in \cite[Theorem 2.44]{Hatcher} that the Euler genus is well-defined for all surfaces. In particular, for a surface $S$ homeomorphic to a sphere with $k \geq 0$ handles, the Euler genus of $S$ is $2k$, and for a surface $S$ homeomorphic to a sphere with $k$ cross-caps, the Euler genus of $S$ is $k$.

We define graph embeddings in the following way. For a graph $G$, we first realize $G$ as a topological space as follows.
We let each vertex of $G$ be represented by a point, and then we let each edge of $G$ be represented by a unit interval. We identify the endpoints of each interval corresponding to an edge $e \in E(G)$ with the points corresponding to the endpoints of $e$, and then we use the standard topology on $G$. In other words, we realize $G$ as a one-dimensional simplicial complex. Then, we define an \emph{embedding} of $G$ in a surface $S$ as a continuous injective map $f: G \hookrightarrow S^{\circ}$ for which $f$ gives a homeomorphism between $G$ and $f(G)$. We refer to the image $f(G)$ as an \emph{embedded graph}, and given a graph $G$ with an embedding function $f: G \hookrightarrow S^{\circ}$,
we will often write $\Tilde G$ for the embedded graph $f(G)$.
If $\Tilde G \in \mathcal G_S$ is an embedded graph with a corresponding graph $G$, then we write $V(\Tilde G)$ for the set of points in $S$ that are the embedded images of vertices of $G$. 
We also often write $|\Tilde G| = |G| = |V(G)|$.
For an embedded graph $\Tilde G \in \mathcal G_S$ with an associated graph $G$ and an embedding function $f: G \hookrightarrow S^{\circ}$, we define an \emph{induced embedded subgraph} of $\Tilde G$ as the image $f(G')$, where $G'$ is an induced subgraph of $G$.
Given an embedded graph $\Tilde G$ in a surface $S$, we define $S \cut \Tilde G$ to be the surface obtained as follows. First, we slightly ``thicken'' $\Tilde G$, which is formalized as adding to $\Tilde G$ an $\varepsilon$-neighborhood of each point of $\Tilde G$ for some small $\varepsilon>0$.
Such a modified $\Tilde G$ forms finitely many path-connected open sets with disjoint boundaries in $S$, and hence, by removing these open sets from $S$, we obtain a surface $S\cut  \Tilde G$.
}

\section{Induced Subgraphs on Surfaces}\label{sec:edgeMaximal}

We will show that for any surface $S$, the surface connectivity of $S$ gives a lower bound for the robust connectivity of any edge-maximal graph embedded in $S$ and furthermore that this lower bound is tight.

First, we establish an upper bound on the surface connectivity of any surface $S$. Since every surface $S$ locally resembles the plane, we may embed $K_4$ on $S$ in such a way that every triangle of $K_4$ separates $S$ into two connected components. Thus, the largest subgraph of $K_4$ that does not separate $S$ contains only two out of the four total vertices, and so $\kappa_s(S) \leq \frac{1}{2}$. When $S$ is the plane, Conjecture \ref{conj:AB} asserts that $\kappa_s(S) = \frac{1}{2}$.

\lv{
Next, we will give an example of an edge maximal graph on any surface $S$ which shows that the optimal lower bound for the robust connectivity of an edge-maximal graph embedded in $S$ cannot be greater than $\frac{1}{2}$. This example will give some insight into the general relationship between surface connectivity and robust connectivity of edge-maximal graphs. We again consider a surface $S$, and we consider a $K_4$ graph with a planar embedding in $S$. By adding a vertex $v_f$ to each triangular face in the planar embedding of $K_4$ in $S$ and then adding an edge from $v_f$ to each vertex of $f$, we obtain an edge-maximal embedding of the triakis tetrahedron. 
In Figure \ref{fig:tri}, we show a triakis tetrahedron $G$ along with a vertex set $R \subseteq V(G)$ for which $\ell (G,R) = \frac{1}{2}$, which shows that the robust connectivity of $G$ is at most $\frac{1}{2}$. In summary, we began with a graph $K_4$, which showed that $\kappa_s(S) \leq \frac{1}{2}$, and by making a small modification, we obtained an edge-maximal graph $G$ embedded in $S$ for which $\kappa_\rho(G) \leq \frac{1}{2}$. We will see that we can use this same strategy to show that the surface connectivity of a general surface $S$ gives a tight upper bound for the robust connectivity of all edge-maximal graphs embedded in $S$.

\sv{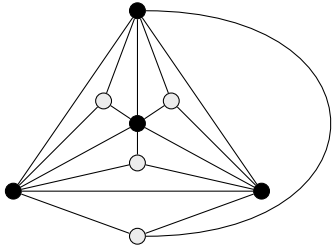
\begin{SCfigure}[1.2][t]}
\lv{
  \begin{figure}
  \begin{center}
  }
\begin{tikzpicture}
[scale=1.5,auto=left,every node/.style={circle,fill=gray!15,minimum size = 6pt,inner sep=0pt}]
\node (r1) at (2,-0.55) [draw = black, fill = gray!15] {};
\node (r2) at (1.7,0) [draw = black, fill = gray!15] {};
\node (v1) at (2,-0.2) [draw = black, fill = black] {};
\node (r4) at (2.3,0) [draw = black, fill = gray!15] {};
\node (r5) at (2,-1.2) [draw = black, fill = gray!15] {};
\node (v2) at (2,0.8) [draw = black, fill = black] {};
\node (v3) at (0.9,-0.8) [draw = black, fill = black] {};
\node (v4) at (3.1,-0.8) [draw = black, fill = black] {};
\draw (r5) to [out=0,in=0,looseness = 02.8] (v2);
\foreach \from/\to in {v1/v2,v1/v3,v1/v4,v2/v3,v2/v4,v3/v4,r1/v3,r1/v4,r1/v1,r5/v3,r5/v4,r2/v1,r2/v2,r2/v3,r4/v1,r4/v2,r4/v4}
    \draw (\from) -- (\to);
\end{tikzpicture}
\lv{\end{center}}
\caption{The graph $G$ shown is the triakis tetrahedron. A vertex set $R \subseteq V(G)$ is shown in black, and $\ell(G,R) = \frac{1}{2}$, which shows that $\kappa_\rho(G) \leq \frac{1}{2}$.
}
\label{fig:tri}
\lv{\end{figure}}
\sv{\end{SCfigure}}
}

\lv{We will need the following lemma.}

\begin{lemma}
\label{lem:connectedSurface}
Let $G$ be a graph, and let $\Tilde G$ be an edge-maximal embedding of $G$ in a surface $S$. Suppose $\Tilde G'$ is a proper induced embedded subgraph of $\Tilde G$ corresponding to an induced subgraph $G' \subseteq G$. If $S \cut \Tilde G'$ is a connected surface, then $ G \setminus  G'$ is a connected graph.
\end{lemma}

\lv{
\begin{proof}
Let $ u,  v \in V( G \setminus  G')$.
Since $S \cut \Tilde G'$ is connected,
we may let $f:[0,1] \rightarrow (S \cut \Tilde G')^{\circ}$ be a path with $f(0) = \Tilde u$ and $f(1) = \Tilde v$. Since $P = f([0,1])$ is compact, we may choose $P$ to be smooth, and by perturbing $P$, we may assume that $P$ only intersects the vertices of $\Tilde G$ at $\Tilde u$ and $\Tilde v$. We may furthermore assume that the values of $t \in (0,1)$ for which $f$ maps $t$ to an embedded edge of $\Tilde G$ are isolated points in $(0,1)$. Since $P$ is compact, this implies that only finitely many values $t \in (0,1)$ are mapped by $f$ to an edge of $\Tilde G$, and we let $t_1, \dots, t_k$ be the set of values in $(0,1)$ that are mapped to an edge of $\Tilde G$ by $f$.

We now construct a walk $W$ on $G \setminus G'$ from $u$ to $v$ using $f$ as follows. We let $W$ begin with $w_0 = u$. Now, for each $i \in \{1,\dots, k\}$, let $\Tilde e_i$ be the embedded edge to which $f(t_i)$ belongs. Iterating through $i \in \{1,\dots, k\}$ in increasing order, we let $w_i$ be an endpoint of $e_i$ not belonging to $G'$. Since $e_i \not \in E(G')$, such a vertex $w_i$ must exist. We claim that $w_i$ is equal to or adjacent to $w_{i-1}$. By construction, the vertices $\Tilde w_{i-1}$ and $\Tilde w_i$ must be incident to a common component of $S \cut \Tilde G$, and it thus is possible to draw an edge $e$ from $\Tilde w_{i-1}$ to $\Tilde w_i$ without introducing a crossing on $S$. Therefore, since $\Tilde G$ is edge-maximal, it must follow that adding $e$ would create a loop or multiple edge in $G$, implying that $w_{i-1}$ and $w_i$ are equal or adjacent, and that we may extend $W$ by appending $w_i$ to its end. Therefore, $u, w_1, \dots, w_k$ is a walk in $G \setminus G'$. Finally, by construction, $\Tilde w_k$ must be incident to a common component of $S \cut \Tilde G'$ with $v$, so by the same argument, $w_k$ is equal to or adjacent to $v$. Therefore, $u, w_0, w_1, \dots, w_k, v$ is a walk from $u$ to $v$ in $G \setminus G'$. Since $u,v$ can be any vertex pair in $G \setminus G'$, it follows that $G \setminus G'$ is connected.
\end{proof}
}

\sv{The proof of the Lemma can be found in the full version of or paper.}
We note that Lemma \ref{lem:connectedSurface} immediately implies that an edge-maximal graph on a surface with at least four vertices is $3$-connected, since in a simple graph, a pair of vertices can only induce an edge, which cannot separate a surface.
We also note that edge-maximality is necessary for Lemma \ref{lem:connectedSurface}. For example, if $\Tilde T$ is an embedded tree in a surface $S$, $S \cut \Tilde T'$ is connected for any induced embedded subgraph $\Tilde T'$ of $\Tilde T$, but $T \setminus T'$ is often disconnected.
Now, we are ready to prove our main result of this section.

\begin{proof}[Proof of Theorem~\ref{thm:SgTriangulation}]
Suppose first that $\kappa_s(S) \geq k$, or in other words, that every embedded graph $\Tilde H \in \mathcal G_S$ on $n$ vertices has an induced embedded subgraph $\Tilde H'$ of size at least $kn$ for which $S \cut \Tilde H'$ is a connected surface.
As shown above using $K_4$, $k \leq \frac{1}{2}$.
Let $G$ be a graph with an edge-maximal embedding on $S$. We note that since every edge-maximal graph $G$ with at most three vertices is a clique and hence satisfies $\kappa_\rho(G) > \frac{1}{2} \geq k$, it suffices only to consider edge-maximal graphs $G$ on at least four vertices. 
In particular, we may assume by Lemma \ref{lem:connectedSurface} that $G$ is $3$-connected.

Now, let $R \subseteq V(G)$. 
If $|R| \leq 3$, then since $G$ is $3$-connected, we may use at least $2 \geq \frac{2}{3}|R| > k|R|$ vertices of $R$ as leaves of some spanning tree on $G$, and we are done in this case.
Now, suppose $|R| \geq 4$. If $G$ has an universal vertex, then we may find a spanning tree on $G$ that uses at least $|R| - 1 \geq \frac{3}{4}|R| > k|R|$ vertices of $R$ as leaves, and we are done. Otherwise,
we consider the graph $\Tilde G[R]$ embedded in $S$. By our hypothesis, we may find a subset $R' \subsetneq R$ of size at least $k|R|$ for which $S \cut \Tilde G[R']$ is a connected surface. We claim that we may find a spanning tree on $G$ that includes every vertex of $R'$ as a leaf. Indeed, as $S \cut \Tilde G[R']$ is connected, and as $G$ is edge-maximal, it follows from Lemma \ref{lem:connectedSurface} that $G \setminus R'$ is a connected graph. Furthermore, since $G$ has no universal vertex,
$G \setminus N(r)$ is a disconnected graph for each $r \in R'$,
so by Lemma \ref{lem:connectedSurface}, $S \cut \Tilde G[N(r)]$
is a disconnected surface.
Therefore, for each $r \in R'$, at least one neighbor of $r$ does not belong to $R'$. 
Hence, one may take any spanning tree $T$ on $G \setminus R'$, and $T$ will dominate $R'$; then one may add each vertex of $R'$ as a leaf of $T$. As $|R'| \geq k|R|$, and as the choice of $R$ was arbitrary, it follows that $\kappa_\rho(G) \geq k$.

Suppose, on the other hand, that every graph $G$ with an edge-maximal embedding on $S$ satisfies $\kappa_\rho(G) \geq k$. 
Let $\Tilde H$ be a graph embedded in $S$.
We seek an induced embedded subgraph $\Tilde H' \subseteq \Tilde H$ of size at least $k |\Tilde H|$ for which $S \cut \Tilde H'$ is a connected surface.
It will make our task no easier to add edges to $\Tilde H$ until $\Tilde H$ is edge-maximal.
Now, let $\Tilde G$ be a graph embedded in $S$ obtained by adding a vertex $v_C$ to each connected component $C$ of $S \cut \Tilde H$ and making $v_C$ adjacent to all vertices incident to $C$.
We call these vertices $v_C$ \emph{component vertices}. Clearly, $\Tilde G$ is an edge-maximal embedding, since every possible edge between vertices of $\Tilde H$ is included, and every possible edge between a component vertex and a vertex of $\Tilde H$ is included. Now, let $R \subseteq V(G)$ denote the vertices that originated in $H$. 
Since $G$ has an edge-maximal embedding, we know that $\kappa_\rho(G) \geq k$, and hence we may choose a spanning tree $T$ on $G$ that includes at least $k|R|$ of the vertices of $R$ as leaves. Then $T' = T \setminus ( \Lambda (T) \cap R)$ is a connected graph that spans all component vertices of $G$. We claim that if $H' =  H[\Lambda (T) \cap R]$, then $S \cut \Tilde H'$ is a connected surface. Indeed, if $S \cut \Tilde H'$ is disconnected,
then there must exist two component vertices of $G$ in distinct connected components of $S \cut \Tilde H'$, and $T'$ cannot contain both of these component vertices, a contradiction.
Therefore, the induced subgraph $\Tilde H'$ is a graph of size at least $k|R| = k|\Tilde H|$, and $S \cut \Tilde H'$ is a connected surface. \sv{\qed}
\end{proof}
While Conjecture \ref{conj:AB} together with Theorem \ref{thm:SgTriangulation} predicts that the surface connectivity of the plane is $\frac{1}{2}$, it is not clear what the surface connectivity of surfaces of higher genus should be. We note that for orientable surfaces $S$ of Euler genus at least $2$ and nonorientable surfaces of Euler genus at least $3$, $K_7$ can be embedded in such a way that any four vertices induce a separating subgraph, which is shown for the torus in Figure \ref{fig:K7}. This implies that $\kappa_s(S) \leq \frac{3}{7}$ for all surfaces except possibly for the plane, the projective plane, and the Klein bottle.

\begin{figure}
  \begin{center}
\begin{tikzpicture}
[scale=1.2,auto=left,every node/.style={circle,fill=gray!30,minimum size = 0pt,inner sep=0pt}]

\node (ll0) at (-0.5,-1) [draw = white, fill = white] {};
\node (ll4) at (2.5,-1) [draw = white, fill = white] {};
\node (ll2) at (1,-1) [draw = white, fill = white] {};
\node (ll15) at (0.5,-1) [draw = white, fill = white] {};
\node (ll1) at (0,-1) [draw = white, fill = white] {};
\node (uu2) at (1,3) [draw = white, fill = white] {};
\node (uu15) at (0.5,3) [draw = white, fill = white] {};
\node (uu1) at (0,3) [draw = white, fill = white] {};
\node (uu0) at (-1.5,3) [draw = white, fill = white] {};
\node (uu05) at (-0.75,3) [draw = white, fill = white] {};
\node (ll05) at (-0.75,-1) [draw = white, fill = white] {};
\node (ur) at (2.5,3) [draw = white, fill = white] {};

\node (rm) at (2.5,1) [draw = white, fill = white] {};
\node (rm15) at (2.5,1.5) [draw = white, fill = white] {};
\node (lm) at (-1.5,1) [draw = white, fill = white] {};
\node (lm15) at (-1.5,1.5) [draw = white, fill = white] {};
\node (lm2) at (-1.5,2) [draw = white, fill = white] {};
\node (lm05) at (-1.5,0.5) [draw = white, fill = white] {};
\node (lm0) at (-1.5,0) [draw = white, fill = white] {};
\node (rm05) at (2.5,0.5) [draw = white, fill = white] {};
\node (rm0) at (2.5,0) [draw = white, fill = white] {};

\node (rm2) at (2.5,2) [draw = white, fill = white] {};
\node (uuu) at (1.75,3)[draw = white, fill = white] {};
\node (lll) at (1.75,-1)[draw = white, fill = white] {};
\node (uuuu) at (1.75,3)[draw = white, fill = white] {};
\node(ll) at (-1.5,-1)[draw = white, fill = white] {};

\node (l1) at (0,0) [draw = black,minimum size = 6pt] {};
\node (l2) at (1,0) [draw = black,minimum size = 6pt] {};

\node (m1) at (-0.5,1) [draw = black,minimum size = 6pt] {};
\node (m2) at (0.5,1) [draw = black,minimum size = 6pt] {};
\node (m3) at (1.5,1) [draw = black,minimum size = 6pt] {};

\node (u1) at (0,2) [draw = black,minimum size = 6pt] {};
\node (u2) at (1,2) [draw = black,minimum size = 6pt] {};

\draw (-1.5,-1) -- (2.5,-1);
\draw (-1.5,3) -- (2.5,3);
\draw (-1.5,-1) -- (-1.5,3);
\draw (2.5,3) -- (2.5,-1);

\foreach \from/\to in {l1/l2,m1/m2,m2/m3,u1/u2,m2/u1,m2/u2,m1/u1,m3/u2,l1/m1,l1/m2,l2/m2,l2/m3,u1/uu1,l1/ll1,l2/ll2,u2/uu2,
ll05/l1,
l1/ll,
lm2/uu05,
m1/lm05,
rm0/lll,
l2/ll15,uu15/u1,lm15/u1,rm05/l2,
u2/uuu,
m3/rm15,
m3/rm2,
m1/lm,m3/rm,m1/lm0,u2/ur}
    \draw (\from) -- (\to);

\end{tikzpicture}
\end{center}
\caption{The figure shows a $K_7$ embedded the torus in such a way that any four vertices induce a face-bounding triangle that separates the surface when removed.}
\label{fig:K7}
\end{figure}
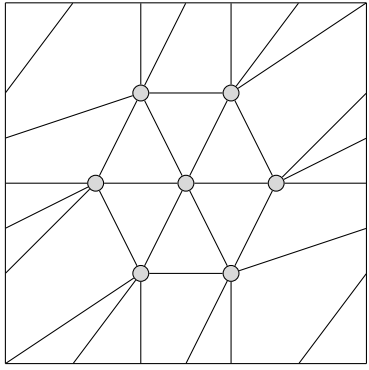

\lv{
\section{Graphs of Bounded Genus}\label{sec:gcgenus}
In this section, we consider embedded graphs on surfaces that are not necessarily edge-maximal, but that are $r$-connected for some value $r \geq 3$. The main goal of this section will be to obtain lower bounds for the robust connectivity of $r$-connected graphs of Euler genus $\gamma$. We will also obtain improved lower bounds for robust connectivity when $\gamma \leq 2$, and we will give examples showing that our bounds are tight within a constant factor.

As a first step in obtaining lower bounds for robust connectivity, we prove a theorem which essentially states that for an $r$-connected graph $G$, with $r \geq 3$, if $R \subseteq V(G)$ is a cutset and $G \setminus R$ does not have too many components, then $R$ has a fairly large subset $R'$ for which $G \setminus R'$ is connected. In our proof of the theorem, we use a greedy procedure in which we iteratively remove from $R$ a vertex with neighbors in the greatest number of distinct components of $G \setminus R$. When the greedy procedure terminates, only one component of $G \setminus R$ remains, and the remaining vertices in $R$ form the set $R'$. A weaker version of the theorem was shown in \cite{Flexibility-arxiv} (Lemma 5.5) using a crude analysis of the same greedy method, but here we give a more careful analysis in order to obtain a better lower bound for the size of $R'$.

In order to estimate how large we can make the set $R'$, we define the following parameter. 
Let $r \geq 3$ be an integer. For real $0 < d \leq 2$, we define $\epsilon_r(d) = 1 - \frac{d}{r}$, and for real $d > 2$, we recursively define $\epsilon_r(d) =  \epsilon_{r}(\lceil d \rceil -1) \left (1 - \frac{1 + d - \lceil d \rceil }{(r-1)(\lceil d \rceil -1)} \right )$. Before we begin to prove that we can obtain a large set $R'$ as described above, we estimate the size of these values $\epsilon_r(d)$. We will only estimate $\epsilon_r(d)$ for integer values of $d$, since we can always use the bound $\epsilon_r(d) \geq \epsilon_r(\lceil d \rceil)$ to get a decent estimate for non-integer values $d$.

\begin{proposition}
\label{rem:sqrt}
For all integers $d \geq 2$ and $r \geq 3$,
$\epsilon_{r}(d) \geq \frac{r - 2}{\sqrt{e} r} (d-2)^{-\frac{1}{r-1}}$.
\end{proposition}
\begin{proof}
We have the recursion $\epsilon_{r}(2) = \frac{r-2}{r}$, $\epsilon_{r}(d) = \left ( 1 - \frac{1}{(r-1)(d-1)} \right ) \epsilon_{r}(d-1)$ for $d \geq 3$. Therefore, for each integer $d \geq 2$, 
\begin{eqnarray*}
\epsilon_{r}(d) &=& \frac{r-2}{r} \cdot \left ( 1 - \frac{1}{2(r - 1) } \right) \left ( 1 - \frac{1}{3(r - 1) } \right) \dots
\left ( 1 - \frac{1}{(d-1)(r - 1) } \right) \\
& > & \frac{r-2}{r} \exp \left ( - \frac{1}{2(r - 1) - 1}  - \frac{1}{3(r - 1)-1} - \frac{1}{4(r - 1)-1} - \dots - \frac{1}{(d-1)(r - 1)-1} \right ) \\
&>&\frac{r-2}{r} \exp \left ( -\frac{1}{r-1} H_{d-2} \right ) ,\\
\end{eqnarray*}
where $H_n$ is the $n$th harmonic number.
Using the inequality $ H_n \leq \log n + 1$, 
\[
\epsilon_{r}(d) 
>\frac{r-2}{r} \exp \left ( - \frac{1}{2} - \frac{1}{r-1} \log (d-2) \right ) \\               
=\frac{r - 2}{\sqrt{e} r} (d-2)^{-\frac{1}{r-1}}. \lv{\hfill \qedhere}
\]
\end{proof}

\begin{theorem}
\label{thm:Hypergraph}
Let $d > 0$ be a real number and $r \geq 3$ an integer.
Let $G$ be an $r$-connected graph, and let $R \subseteq V(G)$ be a subset of vertices such that $G$ is connected even after removing all edges with both endpoints in $R$.
If the number of components of $G \setminus R$ is at most $\frac{d}{r}|R|$,
    then there exists a set $R' \subseteq R$ of at least $ \epsilon_r(d) |R|  $ vertices such that $G \setminus R'$ is a connected graph.    
\end{theorem}

Proving Theorem \ref{thm:Hypergraph} will take some work.
In order to prove the theorem, we will first define a continuous piecewise-linear function $f$, and we will prove certain properties of $f$. Then, we will show a connection between our continuous function $f$ and the discrete-time greedy process in our graph $G$ that we described above, which will ultimately allow us to obtain a lower bound for the size of our set $R'$.

We write $R_0 = |R|$. We define $f:[0,R_0) \rightarrow \mathbb R$ recursively as follows. We let $\alpha_0 = 0$ and let $f_0: \{0\} \rightarrow \frac{d R_0}{r}$ be a mapping from the single point $0$ to $\frac{d R_0}{r}$. It will be convenient to define $\alpha_{-1} = 0$. Now, for $i \in \{1, \dots, \lceil d \rceil -1\}$, 
we define 
\[ \alpha_i =
\begin{cases}
\min \left \{ \alpha_{i-1}  + \frac{rf_{i-1}(\alpha_{i-1}) - (\lceil d \rceil - i)(R_0 - r \alpha_{i-1})}{(r-1)(\lceil d \rceil - i)}, R_0 \right \} & i \in \{1, \dots, \lceil d \rceil - 2\}, \\
 R_0, & i = \lceil d \rceil - 1,
\end{cases}
\]
and 
\begin{eqnarray*}
f_i & : & [\alpha_{i-1},\alpha_i] \cap [0, R_0) \rightarrow \mathbb R \\
f_i(t) &=& f_{i-1}(\alpha_{i-1}) + (- \lceil d \rceil + i)(t - \alpha_{i-1}).
\end{eqnarray*}
Then, we let $f$ be the union of all these mappings, which gives us a piecewise-linear function $f:[0, R_0) \rightarrow \mathbb R$. We note that for each $\alpha_i < R_0$, $f(\alpha_i)$ is defined twice by both $f_i$ and $f_{i+1}$, but these two definitions agree.
Furthermore, if $a_{i-1} = R_0$ for some value of $i$, then $f_{i-1}(\alpha_{i-1})$ will be undefined in the definition of $f_i$, but this is not a concern, because in this case, $f_i$ has an empty domain, and $\alpha_i = R_0$ as well.

Before proceeding, we briefly describe the important properties of $f$. The function $f$ is a strictly decreasing piecewise-linear function. At $t = 0$, the value of $f$ is $\frac{d R_0}{r}$. At first, $f$ decreases with a slope of $-\lceil d \rceil + 1$. Then, after reaching $t = \alpha_1$, the slope of $f$ jumps up to $- \lceil d \rceil + 2$, $f$ and continues to decrease with this new slope. After reaching $t = \alpha_2$, the slope of $f(t)$ again jumps up to $-\lceil d \rceil + 3$, and $f$ continues to decrease. This process continues until reaching $t = \alpha_{\lceil d \rceil - 2}$, at which point the slope of $f$ jumps up a final time to $-1$. Afterward, $f$ decreases at a slope of $-1$ until stopping immediately before $t = R_0$. At this point it is not clear that the slope of $f$ will increase all the way to $-1$ before $t$ reaches $R_0$, but we do not actually need this property for our proof. We sketch the graph of $f$ in Figure \ref{fig:graph}.

We make the following claim about $f$.

\begin{lemma}
\label{lem:difEQ}
The differential equation
$f_+'(t) =  \min \{1 - \lceil \frac{r f(t)}{R_0 - t} \rceil, -1\}$ holds everywhere in $[0, R_0 )$.
\end{lemma}

\begin{proof}
By construction, $f$ is piecewise linear, and hence $f_+'$ is defined everywhere in $[0,R_0)$.
 We define $g(t) = \frac{r f(t)}{R_0 - t}$ so that the differential equation in the lemma states that $f'_+(t) = \min \{1 - \lceil g(t) \rceil , -1\}$. 
 
 We will show that the differential equation in the lemma statement holds by testing that it holds on every interval $[\alpha_{i-1}, \alpha_i)$.
 We will first show that for each $i \in \{0, \dots, i-2\}$, $f$ satisfies the differential equation on the interval $[\alpha_{i-1}, \alpha_i)$.
 We also show that $\lceil d \rceil - i - 1 < g(\alpha_i) \leq \lceil d \rceil - i$. We prove these two statements by induction on $i$. 
 
 When $i = 0$, the interval $[\alpha_{i-1}, \alpha_i)$ is empty, so the first statement is vacuous. Furthermore, it is easy to check that $g(\alpha_0) = g(0) = d$, which is at most $\lceil d \rceil$ and greater than $\lceil d \rceil - 1$. Thus the induction statements hold for $i = 0$.
 Now, let $1 \leq i \leq \lceil d \rceil - 2$. We first show that on the interval $[\alpha_{i-1}, \alpha_i)$, $f_+'(t) = - \lceil d \rceil + i = 1 - \lceil g(t) \rceil$, which will show that the differential equation holds. (Note that for our interval to be nonempty, we must have $\alpha_{i-1} < R_0$.)
 To this end, we show that $g(t)$ is strictly decreasing on this interval. By applying the quotient rule to $g$, we calculate that on this interval, 
 \begin{eqnarray*}
 \frac{(R_0 - t)^2}{r} \cdot g_+'(t) &=& f_+'(t)(R_0 - t) + f(t) \\
  & = &  ( R_0 - \alpha_{i-1}  ) \left ((-\lceil d \rceil + i) + \frac{1}{r} g(\alpha_{i-1}) \right ) \\
  & \leq & (-\lceil d \rceil + i)  ( R_0 - \alpha_{i-1}) \left (\frac{r-1}{r} \right ) < 0.
 \end{eqnarray*}
Hence, since $g(t)$ is strictly decreasing and $\lceil d \rceil - i  < g(\alpha_{i-1}) \leq \lceil d \rceil - i+1$, it suffices to show that if $g(t) = \lceil d \rceil - i$ for some time $t \in (\alpha_{i-1}, \alpha_i]$, then $t = \alpha_i$, since this will show that $f_+' = - \lceil d \rceil + i = 1 - \lceil g(t) \rceil$ everywhere on the interval $[\alpha_{i-1}, \alpha_i)$. To this end, we suppose that $g(t) = \lceil d \rceil - i$ for some $t \in (\alpha_{i-1}, \alpha_i]$. For such a value $t$, we must have
$$g(t) = \frac{r f(t)}{R_0 - t} = \frac{r( f_{i-1}(\alpha_{i-1}) + (-\lceil d \rceil + i)t)}{R_0 - t} = \lceil d \rceil - i.$$
Solving this equation, we see that $t = \frac{rf_{i-1}(\alpha_{i-1}) - (R_0-\alpha_{i-1})( \lceil d \rceil - i)}{(\lceil d \rceil - i)(r - 1)} =  \alpha_i$. Therefore, when $\alpha_{i-1} \leq t < \alpha_i$, $f_+'(t) = 1 - \lceil g(t) \rceil $, which proves the first induction statement. The same argument tells us that $g(\alpha_i) = \lceil d \rceil - i$, which proves the second induction statement.

Finally, when $\alpha_{\lceil d \rceil - 2} \leq t < \alpha_{\lceil d \rceil -1 }$, the same argument tells us that $g'(t)$ is decreasing, so we know that $g(t) < 2$, and therefore, the differential equation holds as long as $f_+'(t) = -1$ everywhere on $[\alpha_{\lceil d \rceil - 2}, \alpha_{\lceil d \rceil - 1})$. However, this follows from the definition of $f$ on $[\alpha_{\lceil d \rceil - 2}, \alpha_{\lceil d \rceil - 1})$. This completes the proof.
\end{proof}

Now that we know how $f_+'$ relates to $f$, we can show that the value of $f$ decreases past $1$ reasonably quickly. We will ultimately use $f$ as an upper bound for the number of components in our graph $G \setminus R$ during our greedy process, so understanding when $f$ reaches $1$ will help us estimate when we only have a single remaining component.

\begin{lemma}
\label{lem:t1}
There exists a value $ t_1 < (1 - \epsilon_r(d))R_0$ for which $f(t_1) = 1$. 
\end{lemma}
\begin{proof}
We observe that our condition on $t_1$ may be equivalently stated as $R_0 - t_1 > \epsilon_r(d) R_0$.

We induct on $\lceil d \rceil$. When $\lceil d \rceil \leq 2$, we see from Lemma \ref{lem:difEQ} that $f_+'(t) = -1$ everywhere, so $f(t) = \frac{d R_0}{r} - t$. Solving $f(t_1) = 1$, we see that $t_1 = \frac{d R_0}{r} - 1 < (1 - \epsilon_r(d))R_0$.

Now, suppose that $\lceil d \rceil \geq 3$. If $f(\alpha_1) < 1$, then the previous argument shows that $t_1 < (1 - \epsilon_r(2))R_0 < (1 - \epsilon_r(d))R_0$. Otherwise, $f(\alpha_1) \geq 1$. We define $\Tilde f(t) = f(t + \alpha_1)$ and restrict $\Tilde f$ to $t  \geq 0$, and we observe that $\Tilde f (0) = R_0 - \alpha_1 := \Tilde{R_0}$. Then, by the induction hypothesis, there exists a value $\Tilde t_1$ for which $\Tilde f(\Tilde t_1) = 1$ and such that \[R_0 - t_1 =  \Tilde R_0 - \Tilde t_1 >  \epsilon_{r}(\lceil d \rceil -1) \Tilde{R_0}  = \epsilon_{r}(\lceil d \rceil -1) \left (1 - \frac{1 + d - \lceil d \rceil }{(r-1)(\lceil d \rceil -1)} \right )R_0 = \epsilon_r(d) R_0.\]
This completes the proof.
\end{proof}

\begin{proof}[Proof of Theorem \ref{thm:Hypergraph}]
We will imagine that the vertices in $R$ are colored red and that all other vertices of $G$ are colored blue, and we refer to the set of connected components of $G\setminus R$ as $B$.
We claim that by using the following greedy process, we obtain a graph with a single blue component and a red vertex set $R' \subseteq R$, such that $|R'|>\varepsilon_r(d) |R|$.
\begin{enumerate}
    \item[($\star$)] While 
    the number of blue components is at least two,
    choose a red vertex $v$ adjacent to the greatest number of distinct blue components, and color $v$ blue.
\end{enumerate}
Since $G$ is at least $3$-connected, coloring all red vertices in $G$ blue yields one blue component; therefore, there exists some finite integer value $t_1$ such that $(\star)$ iterates exactly $t_1$ times and then terminates. For an integer value $0 \leq t \leq t_1$, we define $R_t$ and $B_t$ as the number of red vertices and blue components, respectively, in our graph after $t$ iterations of $(\star)$. In particular, $R_0 = |R|$, and we observe that $R_t= R_0 - t$ for all integer values of $0 \leq t \leq t_1$.
Using $r$, $d$, and $R_0$, we define a function $f(t)$ as in the proof of Theorem \ref{thm:Hypergraph}.
We claim that for all integers $0 \leq t \leq t_1$, after $t$ iterations of $(\star)$,
$B_t \leq f(t)$, and we prove this claim by induction on $t$. %
When $t = 0$, $B_0 \leq \frac{d R_0}{r} = f(0)$.

Now, suppose that the claim holds up to some integer value $t < t_1$; we prove that the claim holds for $t + 1$. Suppose that after $t$ iterations of $(\star)$, $B_t = \frac{k}{r}R_t$ for some real number $k$. Since $G$ is $r$-connected, each current blue component in our graph must be adjacent to at least $r$ red vertices, for a total of at least $B_t r$ 
adjacencies between blue components and red vertices. Therefore, by the pigeonhole principle, there must exist a red vertex adjacent to at least $\lceil \frac{B_t r}{R_t} \rceil = \lceil k \rceil$ distinct blue components,
and therefore, during iteration number $t+1$ of $(\star)$, the number of blue components in our graph decreases by at least $\lceil k \rceil - 1$.
Furthermore, since $G$ is connected even after removing all edges with both endpoints in $R$, we always have a red vertex adjacent to at least two distinct blue components, so $(\star)$ always reduces the number of blue components by at least one.
Therefore, during iteration number $t+1$ of $(\star)$, the number of blue components will decrease by at least $\max \{ \lceil k \rceil -1,1\} =  \max \{ \lceil \frac{rB_t}{R_0 - t} \rceil - 1, 1 \}$.

Now, since $B_t \leq f(t)$, and since $f$ continuously decreases past $1$, there must exist a real value $t^* \geq t$ %
for which $f(t^*) = B_t$.
We claim that $t^*+1 < R_0$. Indeed, by Lemma \ref{lem:t1}, there exists a value $t^*_1$ for which $f(t_1^*) = 1$, and $t_1^* \leq R_0 - 1$. Since $B_t = f(t^*) > 1$, it follows that $t^* < R_0 - 1$, and so $t^* + 1 < R_0$, and so $t^* + 1$ in the domain of $f$.

\begin{figure}
\begin{tikzpicture}
[scale=0.15,yscale=0.25,auto=left,every node/.style={circle,fill=black,minimum size = 0pt,inner sep=1.5pt}]

\node (rm2) at (2.5,2) [draw = white, fill = white] {};

\node (a1) at (34,85) [draw=white,fill=white] {$f(t^*)$};
\node (a1) at (48,60) [draw=white,fill=white] {$f(t^*) + f_+'(t^*)$};
\node (a1) at (25,45) [draw=white,fill=white] {$B_{t+1}$};
\node (a1) at (16,80) [draw=white,fill=white] {$B_{t}$};

\node (a1) at (32,-10) [draw=white,fill=white] { ${\scriptstyle t+1}$};
\node (a1) at (40,-9) [draw=white,fill=white] {${ \scriptstyle t^*+1}$};
\node (a1) at (29,-9) [draw=white,fill=white] {${ \scriptstyle t^*}$};
\node (a1) at (20,-10) [draw=white,fill=white] {${ \scriptstyle t}$};

\draw (0,0) -- (100,0);
\draw (0,0) -- (0,200);

\draw (0,200) -- (10,150);
\draw (10,150) -- (21.25,105);
\draw (21.25,105) -- (34.375,65.625);
\draw  (34.375,65.625) -- (50.78175,32.8125);
\draw  (50.78175,32.8125) -- (100,32.8125 - 50.78175);

\draw (10,155) -- (10, 145);
\draw (21.25,105+5) -- (21.25,105-5);
\draw (34.375,65.625+5) -- (34.375,65.625 - 5);
\draw (50.78175,32.8125 + 5) -- (50.78175,32.8125 - 5);
\draw(100,5) -- (100,-20);


\node (a1) at (10,-30) [draw=white,fill=white] {$\alpha_1$};
\node (a2) at (21.25,-30) [draw=white,fill=white] {$\alpha_2$};
\node (a3) at (34.375,-30) [draw=white,fill=white] {$\alpha_3$};
\node (a4) at (50.78175,-30) [draw=white,fill=white] {$\alpha_4$};
\node (a5) at (100,-30) [draw=white,fill=white] {$\alpha_5$};
\node (t) at (105,-5) [draw=white,fill=white] {$t$};
\node (y1) at (-3.75,10) [draw=white,fill=white] {$1$};


\draw (10,5) -- (10, -20);
\draw (21.25,5) -- (21.25,-20);
\draw (34.375,5) -- (34.375,- 20);
\draw (50.78175,5) -- (50.78175,-20);
\draw (-1.25,10) -- (1.25,10);

\draw[dashed] (0,200) -- (20,80);
\draw[dashed] (20,80) -- (30,50);
\draw[dashed] (30,50) -- (40,30);
\draw[dashed] (40,30) -- (60,10);
\node (a1) at (20,80) [draw=black,fill=black] {};
\node (a2) at (30,50) [draw=black,fill=black] {};

\node (b1) at (25/3+21.25,80) [draw=black,fill=black] {};
\node (b2) at (25/3+21.25+10,80-30) [draw=black,fill=black] {};

\draw[densely dotted] (a1) -- (20,0) ;
\draw[densely dotted] (a2) -- (30,0) ;
\draw[densely dotted] (b1) -- (25/3+21.25,0) ;
\draw[densely dotted] (b2) -- (25/3+21.25+10,0) ;

\draw (20,5) -- (20, -5);
\draw (30,5) -- (30, -5);
\draw (25/3+21.25, 5) -- (25/3+21.25, -5); 
\draw (25/3+21.25+10, 5) -- (25/3+21.25+10, -5);

\draw[densely dotted] (25/3+21.25,80) -- (25/3+21.25+10,80-30) ;

\draw[densely dotted] (0,10) -- (100,10);
\foreach \from/\to in {}
    \draw (\from) -- (\to);

\end{tikzpicture}
\caption{The figure shows an example of the graphs of the continuous function $f(t)$ and the discrete-time function $B_t$ on the $y$-axis, with the parameter $t$ on the $x$-axis. Values of $f(t)$ are shown with a solid curve, and values of $B_t$ are shown with a dashed curve, with non-integer values filled in. The four labelled points on the curves illustrate the estimate $B_{t+1} \leq B_t + f_+'(t^*) = f(t^*) + f_+'(t^*) \leq f(t^* + 1)$ in Theorem \ref{thm:Hypergraph}.
}
\label{fig:graph}
\end{figure}
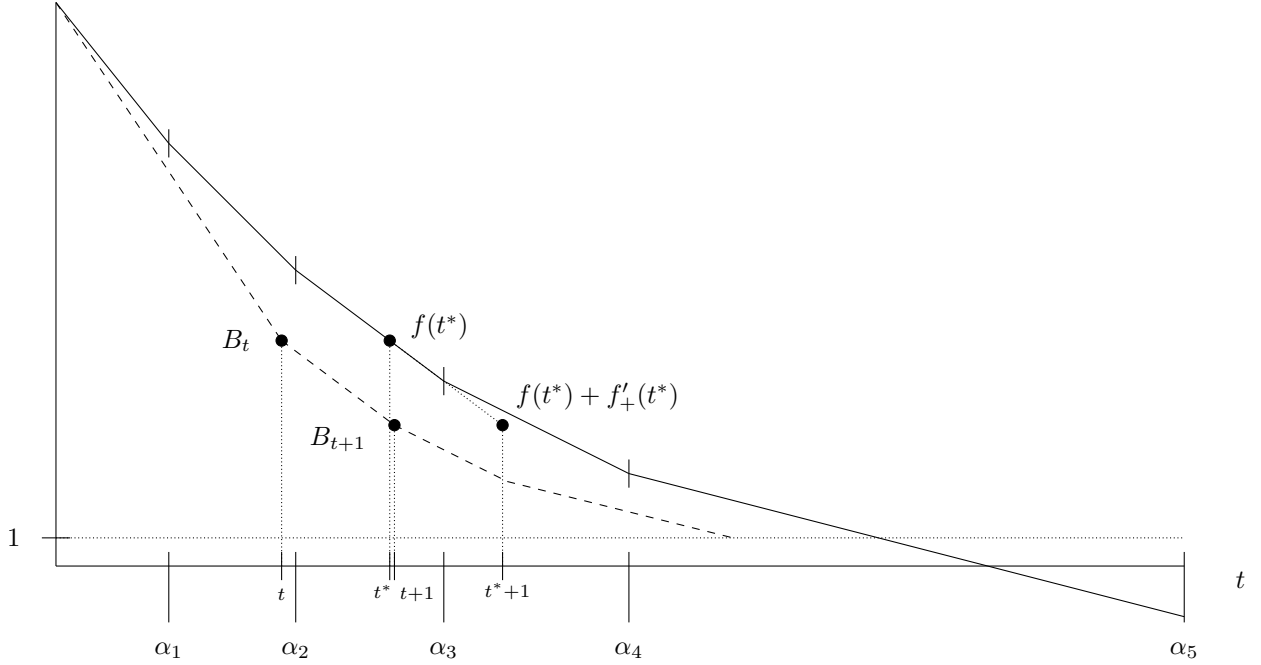

Now, by Lemma \ref{lem:difEQ}, the differential equation $f'_+(t^*) =   \min \{1 - \lceil \frac{rf(t^*)}{R_0 - t} \rceil, -1 \}$ holds. 
Note that here, $f_+'(t^*)$ is the opposite of the quantity $\max \{\lceil \frac{rB_t}{R_0 - t}  \rceil - 1, 1\}$ from the previous paragraph.
Therefore, combining this differential equation with our previous observation about the decrease in the number of blue components, we see that 
\begin{equation}
    \label{estimateB}
    B_{t+1} \leq B_t + f'_+(t^*) = f(t^*) + f'_+(t^*) \leq f(t^* + 1),
\end{equation}
where the last inequality follows from the fact that $f_+'$ is weakly increasing. This completes induction. We illustrate the estimate (\ref{estimateB}) in Figure \ref{fig:graph}.

Now, since $f$ is decreasing and $B_t \leq f(t)$ for integer values $t \leq t_1$, it must follow that $t_1 \leq t^*_1$. Then, by Lemma \ref{lem:t1}, $R_0 - t_1 > R_0 - t_1^* > \epsilon_r(d) R_0$. Therefore, after $t_1$ iterations of $(\star)$, the number of blue components in our graph is one, and the number of red vertices remaining in our graph is at least $\epsilon_r(d) R_0$. We let this remaining set of red vertices make up our set $R'$, which completes the proof.
\end{proof}

Now that the proof of Theorem \ref{thm:Hypergraph} is complete, we are ready to estimate the robust connectivity of $r$-connected graphs of bounded Euler genus.
First, we establish an easy observation that will be useful to estimate some corner cases. Recall that for a graph $G$, $\ell(G,R)$ is the maximum value $k$ for which there exists a spanning tree $T$ on $G$ satisfying $|\Lambda(T) \cap R| = k|R|$.
\begin{observation}\label{obs:smallGraphs}
  If $G$ is an $r$-connected graph and $R\subseteq V(G)$, then $\ell(G,R)\ge \min(\frac{r-1}{|R|},1)$. 
\end{observation}

An easy consequence of Observation~\ref{obs:smallGraphs} is that for an $r$-connected graph on $n$ vertices, $\kappa_\rho(G)\ge \frac{r-1}{n}$.
We will also use the following lemma of Goddard, Plummer, and Swart~\cite{Goddard} about the toughness of graphs of bounded genus in terms of connectivity. While Goddard et al.~originally prove a slightly different lemma for orientable genus, the exact same proof may be used to obtain the following result for Euler genus.
\begin{lemma}[\cite{Goddard}]
\label{lem:Schmeichel}
If $G$ is an $r$-connected graph of Euler genus $\gamma$, and if $X \subseteq V(G)$, then the number of components of $G \setminus X$ is at most $\frac{2}{r - 2}(|X| - 2 + \gamma)$.
\end{lemma}

\begin{theorem}
    \label{thm:k_genus}
 For $r \geq 3$, if $G$ is an $r$-connected graph of Euler genus $\gamma \geq 0$, then $\kappa_\rho(G) \geq \frac{1}{27}(\gamma+1)^{-1/r}$. Furthermore, when $\gamma \leq 2$,
 \begin{itemize}
     \item If $r = 3$, then $\kappa_\rho(G) \geq \frac{21}{256}$.
     \item If $r = 4$, then $\kappa_\rho(G) \geq \frac{5}{27}$.
     \item If $r = 5$, then $\kappa_\rho(G) \geq \frac{49}{192}$.
     \item If $r = 6$, then $\kappa_\rho(G) \geq \frac{3}{10}$.
     \end{itemize}
\end{theorem}
\begin{proof}
  The first inequality is implied by the last four inequalities when $\gamma \leq 2$; therefore, for the proof of the first inequality, we assume that $\gamma \geq 3$. 
We fix a subset $R \subseteq V(G)$, and we aim to show that there exists a spanning tree in $G$ that uses at least $\frac{1}{27}(\gamma+1)^{-1/3}|R|$ vertices from $R$ as leaves.

First, we choose an arbitrary spanning tree $T$ of $G$, and we let $R' \subseteq R$ be the larger of $R \cap U$ and $R \cap W$, where $U$ and $W$ are the two color classes in the bipartition of $T$. 
Clearly, $|R'| \geq \frac{1}{2} |R|$. Furthermore,
$G$ is connected even after all edges with both endpoints in $R'$ are removed. From now on, we will only work with $R'$, and we will ignore all other vertices in $R$.

We set $c = 12 (r - 1) $.
If $|R'|\le c (\gamma+1)^{1/r}$ then we conclude the case by Observation~\ref{obs:smallGraphs}, as in this case we manage to get $\ell(G,R')\ge \frac{r - 1}{c (\gamma+1)^{1/r}}$, and hence $\ell(G,R)\ge \frac{r - 1}{2c}(\gamma+1)^{-1/r} > \frac{1}{27} (\gamma+1)^{-1/r}$.
Otherwise, we color the vertices of $R'$ red, and we color the vertices of each component of $G \setminus R'$ blue.
We give the name $B$ to the set of blue components in $G$.
By Lemma \ref{lem:Schmeichel}, 
\begin{equation}
\label{eqnB}
|B| \leq \frac{2}{r - 2}(|R'| - 2 +  \gamma) = \frac{2}{r - 2} \left (1 + \frac{ (\gamma+1) - 3}{|R'|} \right ) |R'| <  \frac{2}{r - 2} \left ( \frac{1}{c}  (\gamma+1)^{\frac{r-1}{r}}+1 \right )|R'|.
\end{equation}
We observe that $\frac{1}{c} (\gamma+1)^{\frac{r-1}{r}} = \frac{1}{12 (r - 1)}(\gamma+1)^{\frac{r-1}{r}} $ can be bounded below as follows. It follows from Euler's formula that a graph of Euler genus $\gamma$ has a vertex of degree at most $\delta \leq \left \lfloor \frac{5 + \sqrt{24 \gamma + 1}}{2} \right \rfloor$, so $r - 1 \leq \left \lfloor \frac{3 + \sqrt{24 \gamma + 1}}{2} \right \rfloor$. Therefore, 
\[\frac{1}{c} (\gamma+1)^{\frac{r-1}{r}} = \frac{1}{12 (r - 1)}(\gamma+1)^{\frac{r-1}{r}} \geq \frac{1}{12 \cdot \left \lfloor \frac{3 + \sqrt{24 \gamma + 1}}{2} \right \rfloor} (\gamma + 1)^{\frac{r-1}{r}} > \frac{47}{1200} \]
for integer values $\gamma \geq 3$. Therefore, $\frac{1}{c} (\gamma+1)^{\frac{r-1}{r}} + 1 < \frac{1247}{47} \cdot\frac{1}{c} (\gamma+1)^{\frac{r-1}{r}} $, and hence (\ref{eqnB}) gives us
\[ |B| < \frac{2}{(r -2)(r-1) } \cdot \frac{1247}{564} (\gamma +1)^{\frac{r-1}{r}}. \]
Then, by applying Theorem \ref{thm:Hypergraph} with $d = \left \lceil  \frac{2 r}{(r -2)(r-1) } \cdot \frac{1247}{564} (\gamma +1)^{\frac{r-1}{r}} \right \rceil$ and using Remark \ref{rem:sqrt}, 
\[\ell(G,R') > \frac{r-2}{\sqrt{e} r}   \left ( \frac{(r - 2)(r-1)}{2 r} \cdot \frac{564}{1247} \right )^{\frac{1}{r-1}} (\gamma+1)^{-1/r}.\] 
When $r \geq 3$, $(\frac{r - 2}{r})^{\frac{r}{r-1}}(r-1)^{\frac{1}{r -1}} > \frac{27}{100}$, and so \[\ell(G,R') > \frac{27}{100\sqrt{e}} \left ( \frac{564}{2494} \right )^{\frac{1}{r - 1}} (\gamma+1)^{-1/r} >  \frac{2}{27} (\gamma+1)^{-1/r}.\]  for $r \geq 3$. Therefore, $\ell(G,R) \geq \frac{1}{27}$, and the proof of the first lower bound is complete.

Now, when $\gamma \leq 2$, we reconsider the inequality $(\ref{eqnB})$ and observe that 
\[ |B| \leq \frac{2}{r - 2} \left (1 + \frac{ (\gamma+1) - 3}{|R'|} \right ) |R'| \leq  \frac{2}{r - 2} |R'|.\]
Then, for $r \in \{3,4,5,6\}$, we may set $d = \frac{2 r}{r - 2}$ and apply Theorem \ref{thm:Hypergraph}, concluding that $\ell (G, R') \geq \epsilon_{r}(d)$ and hence that $\ell(G,R) \geq \frac{1}{2} \epsilon_{r}(d)$.
\begin{itemize}
\item When $r = 3$, we let $d = 6$, and we conclude that $\ell(G,R) \geq \frac{1}{2} \epsilon_{3}(6) = \frac{21}{256}$.
\item When $r = 4$, we let $d = 4$, and we conclude that $\ell(G,R) \geq \frac{1}{2} \epsilon_4(4) = \frac{5}{27}$. 
\item When $r = 5$, we let $d = \frac{10}{3}$, and we conclude that $\ell(G,R) \geq \frac{1}{2} \epsilon_5(\frac{10}{3}) = \frac{49}{192}$.
\item When $r = 6$, we let $d = 3$, and we conclude that $\ell(G,R) \geq \frac{1}{2} \epsilon_6(3) = \frac{3}{10}$.
\end{itemize}
This completes the proof.
\end{proof}

Theorem \ref{thm:k_genus} tells us that when $r \geq 3$, $r$-connected graphs $G$ of Euler genus $\gamma$ satisfy $\kappa_\rho(G) = \Omega((\gamma + 1)^{-1/r})$.
The following example shows that the lower bound in Theorem \ref{thm:k_genus} is best possible up to a constant factor.

\begin{theorem}\label{thm:constuction}
Let $r \geq 3$ be an integer. For infinitely many values $\gamma \geq 0$, there exists an $r$-connected graph $G$ of genus $\gamma$ satisfying $\kappa_\rho(G) < 4\gamma^{-1/r}$.
\end{theorem}
\begin{proof}
For an integer $n \geq r+1$, consider the Levi graph $G$ of $K_n^{(r)}$, which is shown in Figure \ref{fig:counterexample} with $r = 3$ and $n = 5$. Let the set of red vertices $R \subseteq V(G)$ be the independent set of $n$ vertices of degree ${{n - 1} \choose {r - 1}}$, drawn in black in Figure \ref{fig:counterexample}, and note that $R$ may be obtained as a color class of a spanning tree of $G$. It is straightforward to show that $G$ is $r$-connected, 
and since the Euler genus of $G$ is trivially at most $|E(G)|$, $G$ has Euler genus $\gamma$ satisfying $\gamma \leq r {n \choose r} <  r \left (\frac{ne}{r} \right )^{r}$. Furthermore, by construction, removing any $r$ red vertices from $G$ disconnects the graph, so $\kappa_\rho(G) \leq \frac{r-1}{n} < \frac{e(r - 1)}{r} \left ( \frac{\gamma}{r} \right )^{-1/r} < 4\gamma^{-1/r}$.
\end{proof}

Finally, we give a construction that shows that the lower bound for the robust connectivity of $3$-connected planar graphs in Theorem \ref{thm:k_genus} is correct within a factor of just over $4$.

\begin{theorem}
\label{thm:UB}
For all $\epsilon > 0$, there exists a $3$-connected planar graph $G$ satisfying $\kappa_\rho(G) \leq \frac{1}{3} + \epsilon$.
\end{theorem}
\begin{proof}
Let $H$ be a sufficiently large $3$-regular $3$-edge-connected planar graph.

We construct a graph $G$ and a subset $R \subseteq V(G)$ as follows. For each vertex $v \in V(H)$, we add a triangle $T_v$ to $G$.
Then, for each edge $uv \in E(H)$, we add a vertex $r_{uv}$ to $G$, and we let $r_{uv}$ be adjacent to two vertices of $T_u$ and two vertices of $T_v$.
When adding edges, we require that for each triangle $T_v$, no two vertices $r_{uv}$ and $r_{wv}$ are adjacent to the same vertex pair of $T_v$.
We let each vertex $r_{uv}$ be a member of the set $R$. We may construct $G$ to be planar and $3$-connected. We show an example of a graph $G$ constructed in this way when $H \cong K_4$ in Figure \ref{fig:biggerblowup}, except that in Figure \ref{fig:biggerblowup}, each vertex of $H$ is replaced not with a triangle, but with a larger graph on $36$ vertices. We observe that since $H$ is $3$-regular, $|V(H)| = \frac{2}{3}|E(H)| = \frac{2}{3}|R|$.

Now, consider a subset $R' \subseteq R$, and consider the set $E_{R'} := \{e \in E(H): r_e \in R'\}$. If $G \setminus R'$ is connected, then $E(H) \setminus E_{R'}$ must contain a spanning tree in $H$ and hence must contain at least $|V(H)| - 1 = \frac{2}{3}|R| - 1$ vertices. Therefore, $|E_{R'}| \leq \frac{1}{3}|R| + 1$, and the number of vertices in $R$ that can be removed from $G$ without disconnecting $G$ is at most $\frac{1}{3}|R| + 1$. By letting $|R| = |E(H)| \geq \frac{1}{\epsilon}$, it follows that  $\kappa_\rho(G) \leq \frac{1}{3} + \epsilon$.
\end{proof}

}

\lv{\section{Conclusions}}
\sv{\noindent\textbf{Conclusions.~}}
We have shown tight asymptotic bounds for the robust connectivity of $r$-connected graphs embedded on surfaces\lv{ in Section~\ref{sec:gcgenus}}.
We believe that the notion of robust connectivity is an interesting yet unexplored concept.
\lv{Looking for bounds for robust connectivity in terms of other basic graph properties would be of some interest. }%
Moreover, we show a connection between robust connectivity of edge-maximal graphs and the notion surface connectivity.
We propose a further study of surface connectivity which\sv{,} \lv{connects topological properties of a surface with the properties of graphs drawn on that surface. T}%
\sv{t}%
o the best of our knowledge,
\lv{this research direction }%
has not been considered before.
For planar graphs, the connection we showed provides another equivalent formulation of the famous Albertson Berman conjecture, and our results may give another direction to attack the conjecture itself.
For surfaces of higher genus,
this connection gives rise a more general question. \sv{The graph $K_7$ may be embedded on any surface with at least one handle such that no more than three vertices can be removed without disconnected the surface. Hence, we ask:}%

\begin{question}\label{q:high-surface}
  Is surface connectivity always at least $\frac{3}{7}$ for any surface?
\end{question}
This is a widely open question, and we do not even know whether the correct lower bound decreases with increasing Eulerian genus. \lv{The graph in Figure \ref{fig:K7} shows that if Question \ref{q:high-surface} has an affirmative answer, then $\frac{3}{7}$ is best possible.}

\lv{
We want to reiterate one more question.
For planar triangulations, we conjectured (Conjecture~\ref{conj:K}) that the correct bound for the robust connectivity is $\frac{1}{2}$, which is, if true, the best possible.
For the maxleaf number, a lower bound of $\frac{1}{2}n$ was proved~\cite{Albertson90} for a planar triangulation on $n$ vertices, but it is of great interest to find out what is the correct bound.
We ask whether the trivial upper bound of $\frac{2}{3}n$ given by triangle is always achievable.
\begin{question}\label{q:maxleaf}
  Is the maxleaf number of a planar triangulation on $n$ vertices always at least $\frac{2}{3}n$?
\end{question}
}

\lv{
\bibliographystyle{plainurl}
\bibliography{RobustBib}
}
\sv{

%
%

%
%
%
%
%
%
%

}

\end{document}